 \documentclass[smallabstract,smallcaptions]{dccpaper}

\usepackage{epsfig}
\usepackage{amsmath}
\usepackage{amssymb}
\usepackage{amsthm}
\usepackage{color}
\usepackage{url}
\usepackage{graphicx}
\usepackage{subcaption}

\newcommand{\real}{\mathbb{R}}
\newcommand{\integer}{\mathbb{Z}}

\newcommand{\lam}{\lambda}
\newcommand{\Lam}{\Lambda}
\newcommand{\om}{\omega}

\newcommand{\til}{\tilde}

\newcommand{\mrP}{\mathrm{P}}
\def\inpro#1{\langle#1\rangle}

\newtheorem{thm}{Theorem}
\newtheorem{result}{Result}

\newtheorem{df}{Definition}

\newlength{\figurewidth}
\newlength{\smallfigurewidth}

\begin{document}

\urldef\scat\url{https://github.com/kymatio/kymatio}
\urldef\scatmax\url{https://github.com/TaekyungKi/Scattering_maxp}

\title
{\large
\textbf{Deep Scattering Network with Max-pooling}
}

\author{%
	Taekyung Ki$^{\rm a}$ and Youngmi Hur$^{\rm a,b}$\\[0.5em]
	{\small\begin{minipage}{\linewidth}\begin{center}
				\begin{tabular}{ccc}
					$^{\rm a}$School of Mathematics and & & $^{\rm b}$School of Mathematics\\ 
					Computing (Mathematics)&& Korea Institute for Advanced Study\\
					Yonsei University&& Seoul, 02455, Republic of Korea\\
					Seoul, 03722, Republic of Korea &&\\
					\url{key_@yonsei.ac.kr},  \url{yhur@yonsei.ac.kr} && \\
				\end{tabular}
	\end{center}\end{minipage}}
}

\maketitle
\thispagestyle{empty}

\begin{abstract}
Scattering network is a convolutional network, consisting of cascading convolutions using pre-defined wavelets followed by the modulus operator. Since its introduction in 2012, the scattering network is used as one of few mathematical tools explaining components of the convolutional neural networks (CNNs). However, a pooling operator, which is one of main components of conventional CNNs, is not considered in the original scattering network.

In this paper, we propose a new network, called scattering-maxp network, integrating the scattering network with the max-pooling operator. We model continuous max-pooling, apply it to the scattering network, and get the scattering-maxp network. We show that the scattering-maxp network shares many useful properties of the scattering network including translation invariance, but with much smaller number of parameters. Numerical experiments show that the use of scattering-maxp does not degrade the performance too much and is much faster than the original one, in image classification tasks.
\end{abstract}
\section{Introduction}

Convolutional neural networks (CNNs) \cite{CNN} consist of two steps. The first is the feature extraction from the input by convolution layers and the second is the inference by the fully-connected (FC) layers. Each kernel in the layers would be optimized during the training session. This optimization process is often referred to as \textbf{learning}. Despite the successful application of CNNs \cite{DL}, there are only few tools that provide mathematical explanations of components of the CNNs. 

Scattering network is first introduced by S. Mallat \cite{GIS} in 2012, and since then it is used in many researches. For example, the papers \cite{ISN,SST} replace the feature extraction process in CNNs with the scattering network. The scattering network uses the pre-defined wavelets as the convolution kernels. With pre-defined wavelets, learning does not occur in the feature extraction step so it helps to understand the feature extraction step of CNNs. The scattering network also gives mathematical explanation for translation invariance \cite{GIS} which is a common assumption of CNNs \cite{DL}.
 
The scattering network does not use a pooling operator, which is one of the main components of CNNs. The pooling operator reduces computational complexity, and among many approaches, the max-pooling is the most widely used  \cite{DL}. As the max-pooling is defined on discrete domains, we first propose a \textbf{continuous max-pooling} for continuous domains. We then combine it with the scattering network to obtain the new convolutional network, which we name as \textbf{scattering-maxp network}. The use of the max-pooling reduces the computational cost of the scattering network, while still providing the translation invariance. Numerical experiments on image classification show that our scattering-maxp network reduces the number of parameters and training time of the original one, without too much degradation in performance. Below we summarize some of the notation we use in this paper.

Let $L^2(\real^d) := \{ f : \real^d \to \real : \|f\|_{2}  < \infty \} $, $L^{\infty}(\real^d) := \{ f: \real^d \to \real: \| f\|_{\infty}  < \infty\}$, where $\|f\|_{2} := (\int_{\real^d} |f(x)|^2 dx)^{1/2}$,  and $\| f\|_{\infty} := \inf\{ \alpha >0 : |f(x)| \leq \alpha ~ a.e.~ x \in \real^d\}$.
For $f : \real^d \to \real$, we let ${\rm supp}(f) := {\rm closure}(\{ x \in  \real^d : f(x) \neq 0 \})$. For $f,g \in L^{2}(\real^d)$, the convolution is defined as $(f * g)(y) := \int_{\real^d} f(x) g(y-x) dx$. We use $\hat{f}\in L^{2}(\real^d)$ to denote the Fourier transform of $f \in L^{2}(\real^d)$. For example, when $f$ has a compact support, $\hat{f}(\om) := \int_{\real^d} f(x) e^{- i \inpro{x,\om}} dx$, with $\inpro{x,\om}$ the inner product of $x$ and $\om$ in $\real^d$.

\section{Scattering Network and Scattering-maxp Network}
\subsection{Scattering Network}

In this subsection, we focus on the structure of the scattering network, and refer to \cite{GIS} for more details. The scattering network consists of two parts: feature propagation and low-pass filtering. The feature propagation  is computed by iterative convolutions and the modulus operator where the convolution kernels are the pre-defined wavelets. More precisely, the $d$-dimensional scattering wavelet $\psi_{j,r}$ is defined from the wavelet $\psi \in L^{2}(\real^d)$ by 
\[ \psi_{j, r}(x) := 2^{dj} \psi (2^j r^{-1} x),\]
for an integer $ j > -J $, with fixed $J \in \integer$, and for $r \in G^{+} := G / \{+1, -1\}$, where $G$ is a subgroup of the orthogonal group $O(d)$.  We let $\Lam_{J}^1 := \{( j, r) ~:~ j > -J ~\text{and}~ r \in G^{+}\}$. Let $f \in L^{2}(\real^d)$ be an input. For $p := (\lam_{1} , \lam_{2}, \cdots, \lam_m) \in \Lam_{J}^{m} := \Lam_{J}^1 \times\Lam_{J}^1\times\cdots \times \Lam_{J}^1$ ($m$ times), a \textbf{scattering propagator} $U[p] : L^{2}(\real^d) \to L^{2}(\real^d)$ is given as
$$
	U[p]f(x) :=  U[\lam_m]\cdots U[\lam_2]U[\lam_1]f(x), ~~ \forall f \in L^{2}(\real^d),
$$
where $U[\lam_k]f(x) := |(\psi_{\lam_k} *f)(x)|$ for $k = 1, 2, \cdots, m$ and $U[\emptyset]f(x) := f(x)$.
We call $p\in \Lam_{J}^{m}$ a \textbf{path} of length $m$. In Figure \ref{fg:network}, possible propagations along two different paths $p$ and $q$ are depicted. As we compute $m$ times of convolutions and modulus operators along the path of $p$ of length $m$, the input $f$ is propagated to the $m$-th layer.

\begin{figure}[t]
	\centering
	\includegraphics[width = 1\textwidth]{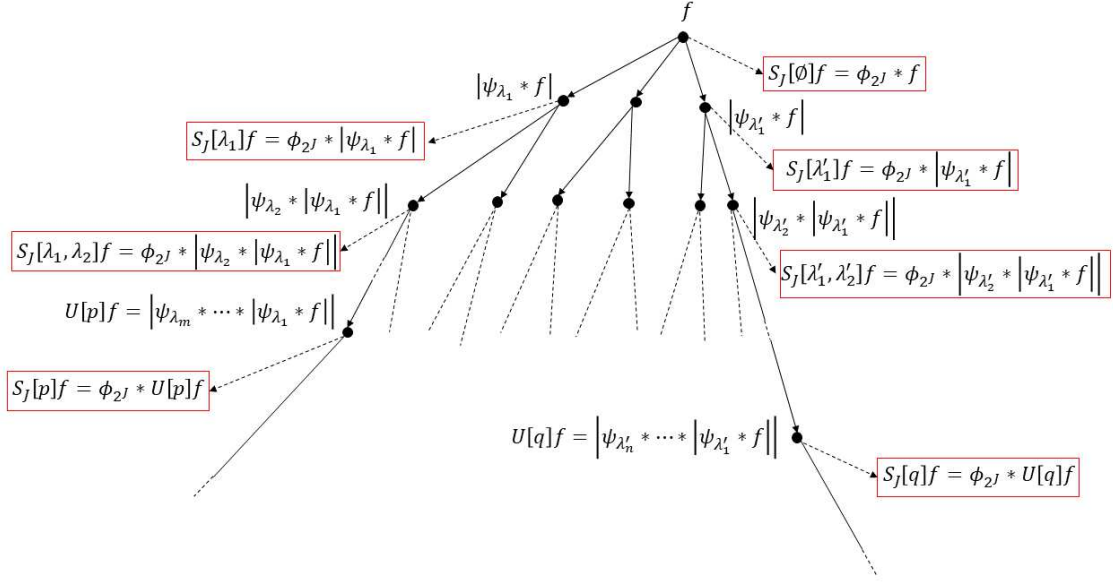}
	\caption{Overview of the scattering network. Let $p :=(\lam_1, \cdots, \lam_m)$ be a path of length $m$ and $q := (\lam'_1,  \cdots, \lam'_{n})$ of length $n$. The input $f$ is propagated to the $m$-th layer along the path $p$ (length $m$) and the $n$-th layer along the path $q$ (length $n$). \label{fg:network}}
	\end{figure}

The output of $U[p]f$ is low-pass filtered with the kernel $\phi_{2^J}(x) := 2^{-dJ}\phi(2^{-J}x)$, which we call \textbf{low-pass kernel} in this paper. A \textbf{windowed scattering transform} $S_J[p] : L^{2}(\real^d) \to L^2(\real^d)$ is defined by
\begin{equation} S_J[p]f(x) := (\phi_{2^J} * U[p]f)(x). \label{eq:window scattering}
\end{equation}
The outputs of the $\{S_J[p]f\}_{p \in \Lam_{J}^m}$ constitute the $m$-th layer of the scattering network. Conditions on  the scattering wavelet $\{ \psi_{\lam} \}_{ \lam \in \Lam_{J}^1 }$ and the low-pass kernel $\phi_{2^J}$ that guarantee certain properties of the scattering network are studied in \cite{GIS}. We omit most of the details about these conditions, but mention the admissibility condition (c.f. Result~\ref{result:TR}) and the condition \eqref{eq:unitary} in Section~\ref{section:TR}.

\subsection{Continuous Max-pooling and Construction of Scattering-maxp Network}
\label{section: scattering-maxp}

Pooling operators are used to effectively reduce the computational cost of the network. Max-pooling operator is one of the most widely used pooling operators, and it outputs maximum values from rectangular neighborhoods (with the same size) of the input~\cite{DL}. %Figure \ref{fg:max-pooling} shows an example of the max-pooling operator applied to an image.

Recall that the original scattering network does not use pooling operators. If we combine the scattering network with pooling operators, the computational cost of the network can be reduced. In a recent paper \cite{MT}, an approach for combining the scattering network with pooling operators such as average pooling and sub-sampling is studied, but the paper assumes pooling operators satisfy certain theoretical property, which does not hold true for max-pooling. This observation is the motivation of our paper. In this subsection, we propose a method for combining the scattering network with max-pooling. Since the scattering network is for the input on continuous domains in $\real^d$, we first model continuous max-pooling operators for continuous domains.

We model the continuous max-pooling for functions $f : \real^d \to \real$ with compact support. As the max-pooling is defined on rectangular domains, we define the rectangular compact region $D$ in $\real^d$ with certain properties, name it as a \textbf{continuous plate}, and model the continuous max-pooling on it.  
\begin{df}%[Continuous Plate]
	Let $\mathcal{A}$ be a collection of $f : \real^d \to \real$ with compact support. If there exist rectangular compact regions $D\subset \real^d$, and $D^{(i)}\subset \real^d$, $i=1, 2, \cdots,N$ satisfying \\
	$(a)\;$ $\cup_{f \in \mathcal{A}}{\rm supp}(f)\subset D$ and $0\in D$,\\
	$(b)\;$ $D = \bigcup_{i = 1}^{N} D^{(i)}$,\\
	$(c)\;$ $\forall i, |D^{(i)}| = |D|/N$,\\
	$(d)\;$ $\exists k \in \real^d \; s.t. \left|D^{(i)} -( D^{(j)} +k )\right| = 0 \text{~if~} i \ne j$, $\;$ ($\;|\cdot|$ denotes the Lebesgue measure),\\
	then we say that $D$ is a continuous plate of $\mathcal{A}$ and $D^{(i)}$ is the $i$-th sub-plate of $D$.
\end{df}

\begin{df}%[Continuous Max-pooling]
	For $f \in L^{\infty}(D)$ and $S > (|D| \|f\|_{\infty}/ \|f\|_{2})^{1/d}$, we define the continuous max-pooling operator $\mrP : L^{\infty}(D) \to L^{\infty}(D/S)$ (with pooling factor $S$) as
	\begin{equation}
	\label{eq:poolingdef}
		\mrP(f)(x) := \sum_{i = 1}^{N} \| f \chi_{D^{(i)}} \|_{\infty} \chi_{D^{(i)}}(Sx).
		\end{equation}
	\end{df}
In the definition and throughout the paper, we let $L^{\infty}(D) := L^{\infty}(\real^d) \cap \mathcal{A}$ when $D$ is a continuous plate of $\mathcal{A}$. The continuous max-pooling $\mrP$ extracts the prominent values from each sub-plate by using the $L^{\infty}$ norm. Since $S >(|D| \|f\|_{\infty}/ \|f\|_{2})^{1/d} \geq 1$, the operator $\mrP$ reduces the plate $D$ to $D/S$. Moreover, using the condition $S >(|D| \|f\|_{\infty}/ \|f\|_{2})^{1/d}$, one can show that (we omit its proof for the sake of brevity)
\begin{equation} \label{eq:Lipschitz}
\forall f \in L^{\infty}(D),\quad \|\mrP(f)\|_{2} \leq \|f\|_{2}.
\end{equation}

Now we construct a new convolutional network. We define a \textbf{pooled scattering propagator} and replace the scattering propagator with it. 
More precisely, we apply the continuous max-pooling operator $\mrP$ to the output of the operator $U[\lam]$.
  \begin{df}%[Pooled Scattering Propagator]
 For a path $p\in  \Lam_{J}^m$, 
  	a pooled scattering propagator $ \til{U}[p] : L^{\infty}(D) \to L^{\infty}(D/S^m)$ is defined by
  	\[ \til{U}[p]f(x) := \til{U}[\lam_m] \cdots \til{U}[\lam_2] \til{U}[\lam_1]f(x), ~~\forall f \in L^{\infty}(D), \]
  	where $\til{U}[\lam_k]f := \mrP(U[\lam_k]f) = \mrP(|\psi_{\lam_k} * f|)$ with the continuous max-pooling $\mrP$.
  	\end{df}
  Similarly as in \eqref{eq:window scattering}, the output of the pooled scattering propagator is finally low-pass filtered with the kernel $\phi_{2^J}$:
  \[ \til{S}_J[p]f(x) := (\phi_{2^J} * \til{U}[p]f)(x), ~~\forall f \in L^{\infty}(D).\]
Figure \ref{fg:networkm} shows an overview of the new network. The input $f$ is propagated to deeper layers by the pooled scattering propagator and is filtered by the kernel $\phi_{2^J}$. For a path $p$ of length $m$, the output of the windowed scattering transform $\til{S}_J[p]f$ constitutes the $m$-th network layer. We call the new network \textbf{scattering-maxp network}.

  \begin{figure}[t]
  	\centering
  	\includegraphics[width=1\textwidth]{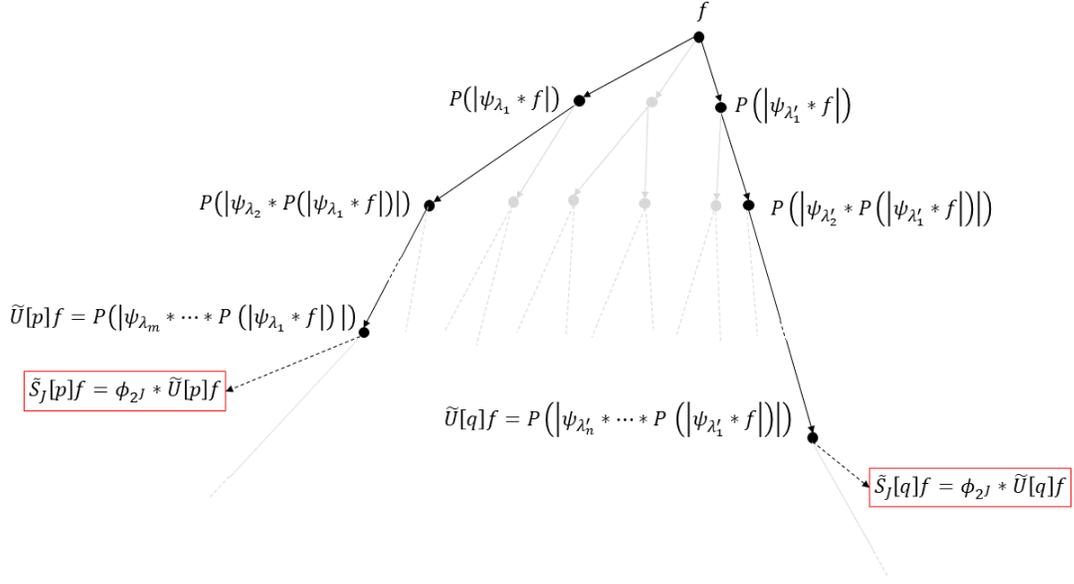}
  	\caption{Overview of the scattering-maxp network. It propagates the input $f$ to deeper layers by the pooled scattering propagators and then its output is low-pass filtered by the kernel $\phi_{2^J}$. We describe propagations along the paths $p$ and $q$. \label{fg:networkm}}
  	\end{figure}

The scattering network and the scattering-maxp network share many useful properties. They use pre-defined wavelets for the convolutional kernels, a propagation is determined by the path, etc. In fact, one can show that both the scattering network and the scattering-maxp network have the translation invariance, albeit in terms of slightly different meaning. This result is presented in the next section.
 
\section{Translation Invariance of Scattering-maxp Network}
\label{section:TR}
In computer vision, translation invariance means that the input and its translation give the same output. For classification tasks, translation invariance would help these two inputs to be in the same class. One common assumption on CNNs is that CNNs have the translation invariance \cite{DL}, but this assumption is usually verified by numerical experiments \cite{translation1,translation2}. On the contrary, it is proved rigorously in \cite{GIS} that the scattering network provides the translation invariance when the scaling level $J$ goes to the infinite. This result is stated below, and in the statement, the admissibility condition of the scattering wavelet is used. We omit its exact definition for the sake of brevity (see (2.27) in \cite{GIS} for the exact definition).

\begin{result}%[Translation Invariance of Scattering Network \cite{GIS}]
	\label{result:TR}
	%Assume that $\Lam_{J}^{m}$ and $S_{J}[p]$, for a path $p\in  \Lam_{J}^m$, are defined as before. 
	If admissible scattering wavelets are used for the scattering network, for any $f\in L^2(\real^d)$ and for any $c \in \real^d$, we have
	\[ \lim_{J \to \infty} \sum_{m=0}^{\infty}\sum_{p \in \Lam_{J}^m}\| S_{J}[p]f - S_{J}[p] L_c f\|_2^2 = 0,\]
	where $L_{c} f = f(\cdot-c)$ is the translation of $f$ by $c$.
\end{result}

One can show that our scattering-maxp network provides the translation invariance as well, under the assumption that the following condition (c.f.~Proposition 2.1 in \cite{GIS}) on the scattering wavelet $\{ \psi_{\lam} \}_{ \lam \in \Lam_{J}^1 }$ and the low-pass kernel $\phi_{2^J}$ holds:
\begin{equation}
	\| \phi_{2^J} * f\|_{2}^2 + \sum_{\lam \in \Lam_{J}^1} \| \psi_{\lam} * f\|_2^2 = \| f\|_2^2 ,\quad \forall f\in L^2(\real^d). \label{eq:unitary}
\end{equation}

\begin{thm}%[Translation Invariance of Scattering-maxp Network] 
	\label{thm:TR}
	Let $f \in L^{\infty}(D)$. Suppose that $c \in \real^d$ satisfies $0 \in D+c$ and that $|\hat{\phi}(\om)||\om| < B$ for some $B>0$ and a.e. $\om \in \real^d$. Then 
	\[ \lim_{m \to \infty}  \sum_{p \in \Lam_{J}^m}\| \til{S}_J [p] f - \til{S}_J[p]T_c f \|^2_2 = 0,\]
	where $T_c : L^{\infty}(D) \to L^{\infty}(D+c)$ is defined by $T_c f(x) := f(x-c).$ \label{thm:translation1}
\end{thm} 
\begin{proof}
Since $ \sum_{p' \in \Lam_{J}^{m+1}} \|\til{U}[p']f \|_2^2 = \sum_{p \in \Lam_{J}^m} \sum_{\lam \in \Lam_{J}^1} \| \til{U}[\lam] \til{U}[p] f\|_2^2$, by \eqref{eq:Lipschitz} and \eqref{eq:unitary}, 
$$\sum_{p' \in \Lam_{J}^{m+1}} \|\til{U}[p']f \|_2^2
\le \sum_{p \in \Lam_{J}^m} ( \| \phi_{2^J}*\til{U}[p]f  \|_2^2 + \sum_{\lam \in \Lam_{J}^1} \|  \psi_{\lam} * \til {U}[p]f \|_2^2 )
=\sum_{p \in \Lam_{J}^m} \| \til{U}[p]f\|_2^2.
$$
%\begin{eqnarray*}
%	\sum_{p' \in \Lam_{J}^{m+1}} \|\til{U}[p']f \|_2^2&\le&\sum_{p \in \Lam_{J}^m} \left( \| \til{U}[p]f * \phi_{2^J}\|_2^2  + \sum_{\lam \in \Lam_{J}^1 } \|\til{U}[\lam] \til{U}[p]f \|_2^2 \right)\\
%	& \leq & \sum_{p \in \Lam_{J}^m} \left( \| \til{U}[p]f * \phi_{2^J} \|_2^2 + \sum_{\lam \in \Lam_{J}^1} \|  \psi_{\lam} * \til {U}[p]f \|_2^2  \right) \\
%	& = & \sum_{p \in \Lam_{J}^m} \| \til{U}[p]f\|_2^2. 
%\end{eqnarray*}
As a result,  $ \sum_{ p \in \Lam_{J}^m} \| \til{U}[p] f \|_2^2 $ is a decreasing sequence in $m$, bounded by $ \|f\|_2^2$. For any $f \in L^\infty(D)$, by \eqref{eq:poolingdef}, we have
\begin{eqnarray}
	\mrP (T_c f)(x)  & = & \sum_{i = 1}^{N} \|(T_cf) \chi_{D^{(i)}+c}) \|_{\infty} \chi_{D^{(i)} +c} (Sx) \nonumber \\ & = & \sum_{i=1}^{N} \| f \chi_{D^{(i)}}\|_{\infty}  \chi_{D^{(i)}}(Sx - c) = T_{c/S}\mrP(f)(x). \nonumber
\end{eqnarray}
Since $\left| 1-e^{-i \inpro{c,\om}/{S^m} } \right| = 2| \sin(\inpro{c,\om}/(2S^m)| \leq |c| |\om|/S^m$ and $|\hat{\phi}(\om)| |\om| < B $, we get
\begin{eqnarray*} 
&{}&\sum_{p \in \Lam_{J}^m} \left\|  \til{S}_J [p] f - \til{S}_J [p] T_c f \right\|_2^2= \sum_{p \in \Lam_J^m} \left\| \widehat{(\til{S}_J[p]f)} - \widehat{(\til{S}_J[p]T_c f)} \right\|_2^2 \\
	& = & \sum_{p \in \Lam_J^m} \left\|\widehat{\phi_{2^J}} \widehat{{(\til{U}[p]f})} (1-e^{-i \inpro{c,\om}/{S^m} } ) \right\|_2^2  
	 \leq  \frac{|c|^2 B^2}{S^{2m}} \sum_{p \in \Lam_J^m} \| \widehat{{\til{U}[p]f}} \|_2^2 
	\leq \frac{|c|^2 B^2}{S^{2m}} \|f \|_2^2.
	\end{eqnarray*}
As the last quantity goes to $0$ as $m \to \infty$, we get the claimed statement.
\end{proof}
Note that we define the translation operator $T_c$ by shifting the function itself and its plate, and this is different from the translation operator $L_c$ used for the scattering network in Result~\ref{result:TR}. Theorem~\ref{thm:TR} says that as the scattering-maxp network gets deeper, the network gets closer to the translation invariant one. In the proof, the condition $S > 1$ is used essentially. Recall that the original scattering network does not use pooling operators. By inserting the pooling operator to the scattering network, we get the translation invariance, which is different from the scattering network.

\section{Numerical Experiments}
\label{section:Exp}

\begin{figure}[t]
	\centering
	\begin{subfigure}[t]{1\textwidth}
		\centering
		\includegraphics[width=0.9\textwidth]{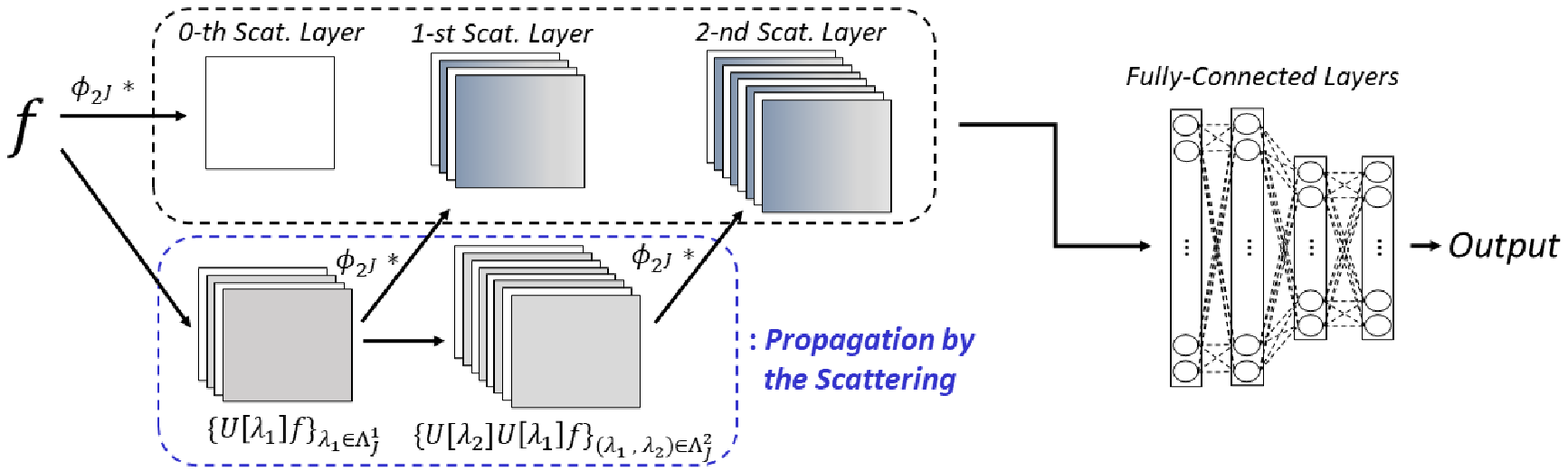}
		\caption{\textbf{Scattering} (or \textbf{Original scattering}). }
	\end{subfigure}%
	
	\begin{subfigure}[t]{1\textwidth}
		\centering
		\includegraphics[width=0.9\textwidth]{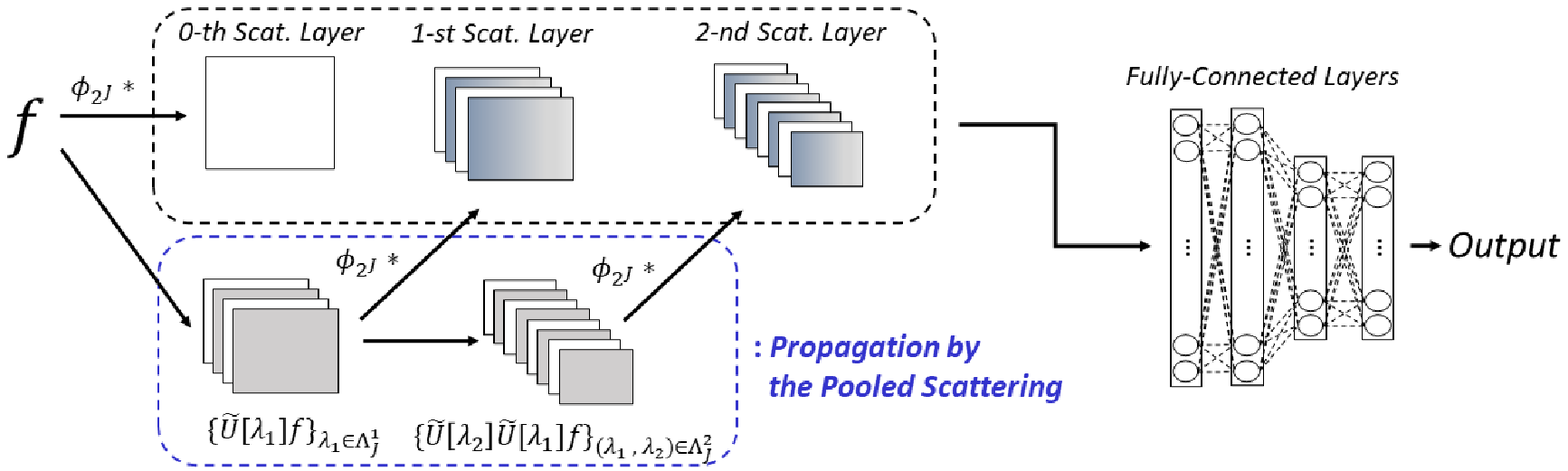}
		\caption{\textbf{Scattering-maxp}.}
	\end{subfigure}
	\begin{subfigure}[t]{1\textwidth}
		\centering
		\includegraphics[width=0.9\textwidth]{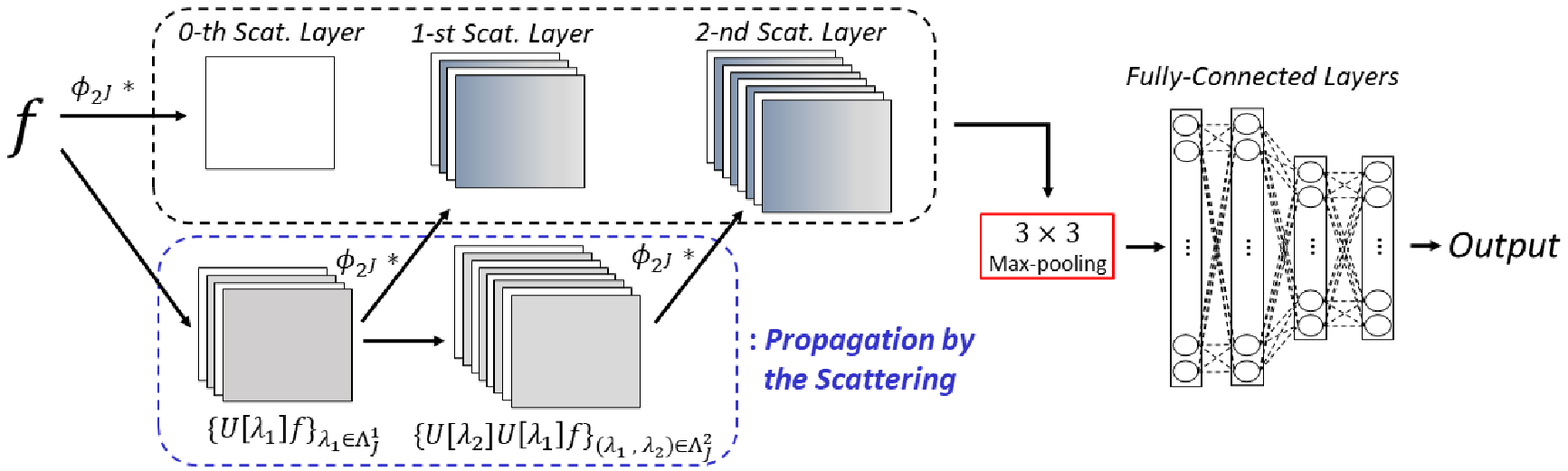}
		\caption{\textbf{Scattering-naivep}.}
	\end{subfigure}
	\caption{Three scattering based models for image classification experiments. \label{fg:experiment model}}
\end{figure}

In this section, we build three scattering network based models and evaluate them by numerical experiments on image classification. We also compare these models with three well-known CNNs (without explanation about their structures): VGG-16 \cite{VGG}, ResNet34 \cite{ResNet}, MobileNet \cite{MobileNet}. 
We use Caltech-101 \cite{Caltech-101} and Caltech-256 \cite{Caltech-256} databases for our experiments. We follow the set-up in \cite{kymatio} for the scattering wavelets and low-pass kernel. We use the Morlet wavelet \cite{Morlet} for $\psi$, and the Gabor filter \cite{Gabor} for $\phi$. For the exact set-up of scattering wavelets and low-pass kernel, see \cite{ISN,kymatio}. 

\subsection{Training Details and Model Preparations}
\label{section:Preparations}
For all of our experiments, the following training details are applied.
\begin{itemize}\setlength\itemsep{-0.1cm}
	\item CPU: Intel(R) Xeon(R) Gold 5210 @ 2.20GHz / GPU: Tesla V100-32GB.
	\item TensorFlow 1.15.1 / Keras 2.2.4-tf.
	\item Loss: Categorical cross-entropy / Optimizer: Adam
	\item Data Augmentations: Horizontal flip \& 8 rotations from $-20^{\circ}$ to $20^{\circ}$.
	\item Batch Size: 256
	\item Ratio of the Training Data and the Validation Data: 3 to 1.
	\item Input Size: $224 \times 224$. 
\end{itemize}
Figure \ref{fg:experiment model} depicts three scattering based models. We call the first model \textbf{scattering}\footnote{\scat}. For an input $f$, we compute $\phi_{2^J} * U[p]f$ along the path $p$ of length $m$, where $m = 0, 1, 2$, so we compute three scattering layers. These scattering layers are concatenated and then connected to the $4$ FC layers with ReLU nonlinear operator. The first two FC layers have $512$ units and the rests have $256$ units. We obtain the output at the end computed by Soft-max. 
We call the second model \textbf{scattering-maxp}\footnote{\scatmax}. We use pooled-scattering propagators by inserting the $2 \times 2$ max-pooling (with $2\times 2$ strides) to the output of the scattering propagators. The rest of the model is same as the first model. For the third model, we consider applying max-pooling naively (without theoretical ground) to the scattering network, and call the resulting model \textbf{scattering-naivep}. This model has the same structure as the first model, but  $3 \times 3$ max-pooling (with $3 \times 3$ strides) is added right after the concatenation of three scattering layers.

\subsection{Image Classification Results}
\begin{figure}[t]
	\centering
	\includegraphics[width = 1\textwidth]{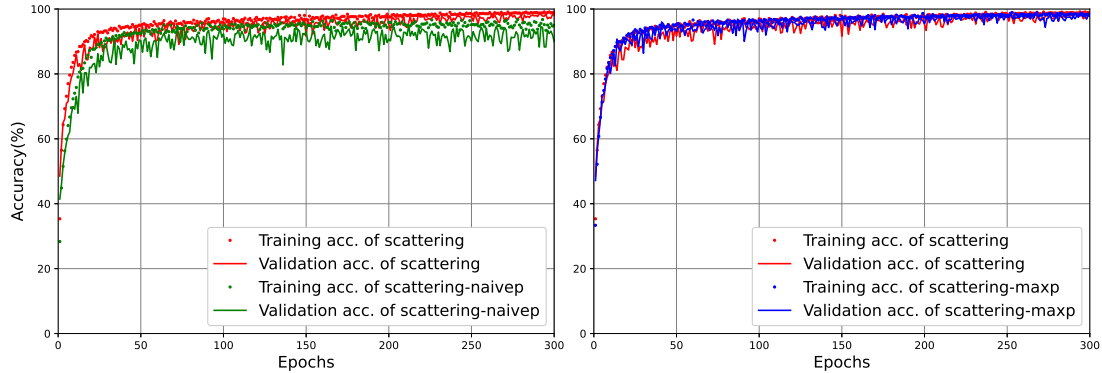}
	\caption{Performance of the scattering based models on Caltech-101 (300 epochs). The learning rate is $10^{-3}$ during the whole training session. \label{fg:Caltech-101}}
\end{figure}

\begin{table}[tp]
	\begin{center}
		\caption{\label{tab:Caltech-101}%
			Classification results of Caltech-101.}
		{
			\renewcommand{\baselinestretch}{1}\footnotesize
			\begin{tabular}{c|c|c|c}
				\multicolumn{1}{c|}{\textbf{Model}} & \multicolumn{1}{|c|}{\textbf{$\#$ of parameters}} & \multicolumn{1}{|c|}{\textbf{Accuracy} (\%)} & \multicolumn{1}{|c}{\textbf{Training time} (s/epochs)} \\
				\hline
				\multicolumn{1}{l|}{VGG-16} & \multicolumn{1}{|c|}{\small 134,677,286} & \multicolumn{1}{|c|}{99.58} & \multicolumn{1}{c}{566}  \\
				\multicolumn{1}{l|}{ResNet34} &\multicolumn{1}{|c|} {21,344,166} & \multicolumn{1}{|c|}{99.21}&  \multicolumn{1}{|c}{221}  \\
				\multicolumn{1}{l|}{MobileNet} &\multicolumn{1}{|c|}{\textbf{3,332,742}} & \multicolumn{1}{|c|}{\textbf{99.95}} & \multicolumn{1}{|c}{419} \\
				\hline
				\multicolumn{1}{l|}{Scattering} &\multicolumn{1}{|c|}{87,592,038} & \multicolumn{1}{|c|} {98.49} & \multicolumn{1}{|c}{284} \\
				\multicolumn{1}{l|}{Scattering-naivep} &\multicolumn{1}{|c|}{11,596,902} & \multicolumn{1}{|c|} { 94.54} & \multicolumn{1}{|c}{267} \\
				\multicolumn{1}{l|}{Scattering-maxp} & \multicolumn{1}{|c|}{9,944,166} & \multicolumn{1}{|c|} {98.59} & \multicolumn{1}{|c}{\textbf{206}}
		\end{tabular}}
	\end{center}
\end{table}

\begin{figure}[t]
	\centering
	\includegraphics[width = 1\textwidth]{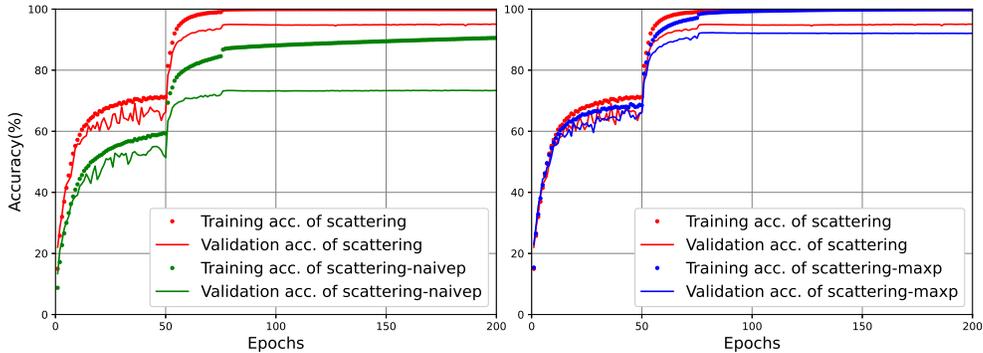}
	\caption{Performance of scattering based models on Caltech-256 (200 epochs). The learning rate is adjusted as $10^{-3}$ initially, $10^{-4}$ after 50, and $10^{-5}$ after 75 epochs. \label{fg:Caltech-256}}
\end{figure}
\begin{table}[tp]
	\begin{center}
		\caption{\label{tab:Caltech-256}%
			Classification results of Caltech-256.}
		{
			\renewcommand{\baselinestretch}{1}\footnotesize
			\begin{tabular}{c|c|c|c}
				\multicolumn{1}{c|}{\textbf{Model}} & \multicolumn{1}{|c|}{\textbf{$\#$ of parameters}} & \multicolumn{1}{|c|}{\textbf{Accuracy} (\%)} & \multicolumn{1}{|c}{\textbf{Training time} (s/epochs)} \\
				\hline
				\multicolumn{1}{l|}{VGG-16} & \multicolumn{1}{|c|}{135,312,321} & \multicolumn{1}{|c|}{98.58} & \multicolumn{1}{|c}{1926}   \\
				\multicolumn{1}{l|}{ResNet34} &\multicolumn{1}{|c|} {21,423,681} & \multicolumn{1}{|c|}{99.97}&  \multicolumn{1}{|c}{750}  \\
				\multicolumn{1}{l|}{MobileNet} &\multicolumn{1}{|c|}{\textbf{3,491,617}} & \multicolumn{1}{|c|}{\textbf{99.99}} & \multicolumn{1}{|c}{1433}  \\
				\hline
				\multicolumn{1}{l|}{Scattering} &\multicolumn{1}{|c|}{87,631,873} & \multicolumn{1}{|c|} {95.06} & \multicolumn{1}{|c}{1006} \\
				\multicolumn{1}{l|}{Scattering-naivep} &\multicolumn{1}{|c|}{11,636,737} & \multicolumn{1}{|c|} {73.38} & \multicolumn{1}{|c}{982} \\
				\multicolumn{1}{l|}{Scattering-maxp} &  \multicolumn{1}{|c|}{9,984,001} & \multicolumn{1}{|c|} {92.11} & \multicolumn{1}{|c}{\textbf{651}}
		\end{tabular}}
	\end{center}
\end{table}

\noindent\textbf{Caltech-101:} Caltech-101 \cite{Caltech-101} consists of 9,146 images of insects, animals, foods, etc. within $102$ categories. As we apply the set-up in Subsection \ref{section:Preparations}, we conduct classification tasks on 116,586 training data and 38,862 validation data. Figure \ref{fg:Caltech-101} shows the performance of the scattering based models. The original scattering achieves about $98 \%$ accuracy. The scattering-naivep achieves about $94\%$ accuracy, which shows the naive max-pooling reduces the performance. In contrast, the scattering-maxp achieves competitive accuracy with the original one. Table \ref{tab:Caltech-101} shows the best results of these models with other CNN models. The three CNNs achieve slightly higher accuracy than the scattering based models. The MobileNet achieves the highest accuracy among all models, but our scattering-maxp has the fastest training time.

\noindent\textbf{Caltech-256:} Caltech-256 \cite{Caltech-256} consists of 30,607 images of animals, foods, insects, machines etc. within $257$ categories. Using the set-up in Subsection \ref{section:Preparations}, we conduct classification tasks on 390,239 training data and 130,080 validation data. Figure \ref{fg:Caltech-256} shows the performance of the scattering based models. The scattering achieves about $95\%$ accuracy. The scattering-naivep achieves about $73\%$ accuracy, which shows inserting the max-pooling naively yields drastic decrease of accuracy. But our scattering-maxp achieves about $92 \%$ accuracy. Table \ref{tab:Caltech-256} shows the best results of these models with other CNN models. Similar to the results on Caltech-101, MobileNet ranks the first in accuracy and our scattering-maxp ranks the first in training time. 

\section{Conclusions}
We modeled continuous max-pooling for continuous domain, combined it with the scattering network, and constructed the new convolutional network, called scattering-maxp network. We proved that the scattering-maxp network provides the translation invariance as the network gets deeper. We showed, by numerical experiments on image classification, that the scattering-maxp network can replace the convolutional layers in CNNs without too much loss of accuracy and with reduced training time.

\section*{Acknowledgment}
This work was supported in part by the National Research Foundation of Korea (NRF) [Grant Number  2015R1A5A1009350], and
was in part the result of a study on ``HPC Support" Project, supported by the `Ministry of Science and ICT’ and NIPA.
\section*{References}
\bibliographystyle{IEEEbib}
\bibliography{ms}

\end{document}